\documentclass[12pt]{article}
\usepackage[OT2,OT1]{fontenc}
\usepackage{amsfonts,amsmath, amssymb,latexsym,mathrsfs}
\usepackage[all,cmtip]{xy}

\def\PGL{{\rm PGL}}

\def\Stab{{\rm Stab}}

\def\Gal{{\rm Gal}}

\def\P{{\mathbb P}}
\def\Disc{{\rm Disc}}

\def\irr{{\rm irr}}
\def\Vol{{\rm Vol}}

\def\FF{{\mathcal F}}
\def\RR{{\mathcal R}}

\def\P{{\mathbb P}}

\newcommand{\Z}{\mathbb Z}
\newcommand{\R}{\mathbb R}
\newcommand{\Q}{\mathbb Q}
\newcommand{\C}{\mathbb C}

\newcommand{\HH}{\mathbb H}
\newcommand{\GL}{{\rm GL}}
\newcommand{\SL}{{\rm SL}}
\newcommand{\SO}{{\rm SO}}

\newtheorem{theorem}{Theorem}
\newtheorem{corollary}[theorem]{Corollary}
\newtheorem{lemma}[theorem]{Lemma}
\newtheorem{proposition}[theorem]{Proposition}
\newtheorem{remark}[theorem]{Remark}

\newenvironment{proof}{\noindent {\bf Proof:}}{$\Box$ \vspace{2 ex}}
\title{On the number of integral binary $n$-ic forms \\having bounded Julia invariant}
\author{Manjul Bhargava \and Andrew Yang}
\begin{document}
\maketitle

\abstract{In 1848, Hermite introduced a reduction theory for binary
  forms of degree $n$ which was developed more fully in the seminal
  1917 treatise of Julia. This canonical method of reduction made use
  of a new, fundamental, but irrational $\SL_2$-invariant of binary
  $n$-ic forms defined over $\R$, which is now known as the Julia
  invariant.  In this paper, for each $n$ and $k$ with $n+k\geq 3$, we
  determine the asymptotic behavior of the number of
  $\SL_2(\Z)$-equivalence classes of binary $n$-ic forms, with $k$
  pairs of complex roots, having bounded Julia invariant.
  Specializing to $(n,k)=(2,1)$ and $(3,0)$, respectively, recovers
  the asymptotic
  results of Gauss and Davenport on positive definite binary quadratic
  forms and positive discriminant binary cubic forms, respectively.}

\section{Introduction}

Let $V_n(\R)$ denote the $(n+1)$-dimensional real vector space of binary $n$-ic
forms \begin{equation}\label{bnic} 
f(x,y)=a_0x^n+a_1x^{n-1}y+\cdots+a_ny^n\end{equation} 
having coefficients $a_0,\ldots,a_n\in \R$.  
The group $\SL_2(\R)$ acts naturally on $V_n(\R)$ via linear substitution of variable;  namely, an element 
$\gamma \in \SL_2(\R)$ acts on $f(x,y)$ by

\begin{equation}
  \label{equntwistedaction}
\gamma\cdot f(x,y)=f((x,y)\cdot \gamma).  
\end{equation}
This action of $\SL_2(\R)$ on $V_n(\R)$ is a left action, i.e.,
$(\gamma_1\gamma_2) \cdot f=\gamma_1\cdot(\gamma_2\cdot f)$.

In 1917, Julia~\cite{Julia} introduced a natural invariant $\theta(f)$
for this action of $\SL_2(\R)$ on binary $n$-ic forms.  The invariant
was constructed in terms of the discriminant of a certain canonical
but {irrational} positive-definite $\SL_2(\R)$-covariant binary
quadratic form $Q$ of $f$.
More precisely, consider a binary $n$-ic form $f$ with coefficients as in 
(\ref{bnic}).  If $a_{0} \neq 0$, we may write
\[ f(x, y) = a_{0} (x - \alpha_{1}y)(x - \alpha_{2}y) \cdots (x - \alpha_{n}y) \]
with $\alpha_{i}\in\C$,
and then, for any vector $t=(t_1,\ldots,t_n)$ of positive real numbers, we
may consider the positive-definite quadratic form
\begin{equation}                      
 Q_t(x, y) = \sum_{j=1}^{n} t_{j}^{2} (x - \alpha_{j} y) (x - \overline{\alpha}_{j} y).
\end{equation}
Julia chose the $t_{j}$ so as to minimize the expression
\begin{equation} \label{def:theta}           
 \theta = \theta(f) = \frac{a_{0}^{2} \, |\Disc \, Q_{t}|^{n/2}}{t_{1}^{2} \ldots t_{n}^{2}}
\end{equation}
and proved that with this restriction on the $t_{j}$, the form $Q(x,
y):=Q_t(x,y)$ is a covariant of the original form $f(x, y)$.
Julia also proved that the resulting expression (\ref{def:theta}) for
$\theta(f)$ is then an $\SL_2(\R)$-invariant of the binary form $f$.
We call the quantity $\theta = \theta(f)$ the \emph{Julia invariant}
of the binary form $f(x, y)$.

Julia showed that the quadratic covariant $Q(x,y)$ enables one to give
a natural $\SL_{2}(\Z)$-reduction theory for binary $n$-ic forms over
$\R$ (or over $\Z$); namely, one says that $f$ is \emph{reduced} if
$Q$ is reduced, in the usual sense of Gauss, as a positive-definite
binary quadratic form.  Furthermore, Julia proved that $\theta$ bounds
many quantities of interest for a reduced form $f(x, y)$; for example,
the leading coefficient $a_{0}$ is bounded by a constant times
$\sqrt\theta$, while the roots $\alpha_{i}$ of a reduced form are
bounded by a constant times $\sqrt{\theta}/|a_{0}|$.  Julia's reduction
theory has been implemented to great effect in Cremona's
work~\cite{Cremona} for cubic and quartic forms (for the purpose of
efficient descent on elliptic curves), and in the work of
Stoll and Cremona~\cite{CS} for forms of general degree.

Given the naturality and utility of the Julia invariant of binary forms,
the question arises:
how many $\SL_2(\Z)$-equivalence classes of integral binary $n$-ic forms are there
having Julia invariant at most $X$, as $X$ gets large?
More precisely, let $N_{n,k}(X)$ denote the number of
$\SL_{2}(\Z)$-equivalence classes of integral irreducible binary $n$-ic forms, having $k$
pairs of complex roots and $n-2k$ real roots, such that $\theta(f)
\leq X$.  In~\cite{ayang:thesis}, using the estimates of Julia as well
as some additional input from the paper~\cite{CS}, it was shown that
$N_{n,k}(X) = O_{\varepsilon}(X^{\frac{n+1}{2} + \varepsilon})$, for
any $\varepsilon > 0$.  
The primary objective of this article is to
refine the latter estimate to an exact asymptotic, along with a power-saving
error term.  Specifically, we prove the following theorem:

\begin{theorem} \label{refbq} Let $n$ and $k$ be non-negative integers
  with $k\in\{0,1,\ldots,\lfloor n/2\rfloor\}$ such that $n+k\geq 3$.  Then there exists a
  constant $c_{n,k}>0$ such that $$N_{n,k}(X)\,=\,c_{n,k}
  X^{\textstyle \frac{n+1}{2}} + O(X^{{\textstyle\frac{n+1}{2}-}{\textstyle \frac1n}}).$$
\end{theorem}

Since the Julia invariant coincides with the discriminant and
squareroot of the discriminant in the cases of binary quadratic forms
having two complex roots and binary cubic forms having three real
roots, respectively, the above theorem includes and extends the
Gauss class number summation formula for binary quadratic
forms of negative discriminant~\cite[Art.~302]{Gauss}, and Davenport's theorem on the density
of discriminants of binary cubic forms of positive
discriminant~\cite{Davenport}; the above-stated error terms in these two cases were first proven by Shintani, as second-order terms, in \cite{Shintani} and \cite{Shintani2}, respectively.  Theorem~\ref{refbq} thus gives a natural
way to count and enumerate $\SL_2(\Z)$-equivalence classes of 
integral binary $n$-ic forms for any degree $n$ and any signature, in a uniform manner, 
extending the results and methods already known for binary quadratic and cubic forms. 
One recent application of Theorem~\ref{refbq} and the methods behind its proof is
seen in the beautiful work of Ho, Shankar, and Varma~\cite{HSV},
where it is shown that there are $S_n$-number fields of {every} odd degree $n$ having odd class number. 

As shown in Section~\ref{avgsec}, the constant $c_{n, k}$ is the value of a certain
integral over a fundamental region. We do not carry out the computation, but for 
$n+k\leq 3$ this value is known; Gauss~\cite{Gauss} showed that~$c_{2,1}=\frac{\pi}{36}$, while
Davenport~\cite{Davenport} showed that $c_{3,0}=\frac{\pi^{2}}{36}$.

Another natural question is whether analogous results for binary $n$-ic forms are known for other
invariants, particularly rational invariants.  As
mentioned above, in the $(n, k) = (2, 1)$ and $(3, 0)$ cases, the Julia variant
is essentially the discriminant, and the asymptotics in these cases were known to Gauss and Davenport,
respectively.  For general binary $n$-ic forms, Birch and Merriman~\cite{BM}
proved that the number of binary $n$-ic forms having a fixed discriminant is 
finite. Their result was ineffective, while the first effective bound was
proven by Evertse and Gy\H{o}ry~\cite{EG}.  It is expected that the number of binary $n$-ic forms
 having absolute discriminant less than $X$ should be asymptotic to $d_n X^{(n+1)/(2n-2)}$ for some
 constant $d_n>0$; however, the best known bounds are currently exponential in~$X$. 
In~the case of binary quartic forms, the ring of polynomial invariants is generated by two invariants
commonly denoted $I$ and $J$ (in particular, the discriminant is a polynomial in $I, J$).  
In~\cite{BS}, the first-named author and Shankar proved asymptotics of the form
\begin{equation} \label{quarticIJcount}
\sum_{\max\{|I|^{3}, J^{2}/4\}<X} h(I, J)\,=\,\frac{4}{135} \zeta(2) X^{5/6} + O(X^{3/4 + \varepsilon})
\end{equation}
where $h(I, J)$ denotes the number of classes of irreducible binary quartic forms
 having invariants $I$ and $J$ and four real roots; similar asymptotics with different constants in the main
term were also obtained for the other two possible real signatures.

The organization of this article is as follows. In Section~\ref{prelim}, we
review some of the basic facts about the $\SL_2(\R)$-covariants $Q$
and $\theta$.  In Sections~\ref{cpt} and \ref{redthbq}, we establish
some convenient fundamental domains for the actions of $\SL_2(\Z)$ on
$V_n(\R)$. As in the classical works of Gauss and Davenport, 
the primary difficulty in counting points with bounded Julia invariant in these fundamental
domains is that they are not compact, but instead have a cuspidal
region going off to infinity.  To deal with and effectively handle
this cusp, in Section~\ref{redsec} we investigate the distribution of
reducible and irreducible points inside these fundamental
domains. Specifically, we prove that the cusp contains only reducible points,
while the remainder of the domain outside the cuspidal region contains primarily irreducible
points with Galois group $S_n$ and, when $n\geq 5$, having trivial stabilizer. In
Section~\ref{avgsec}, we then develop
a refinement of an averaging method introduced in~\cite{dodpf}
to count irreducible points of bounded Julia invariant in these
fundamental domains in terms of the volumes of these domains, via arguments that
work uniformly in the degree $n$.  This
then allows us to prove the asymptotic formula contained in Theorem~1.
Finally, in Section~\ref{bfcong}, we prove a stronger version of
Theorem~\ref{refbq} where we restrict to counting those binary $n$-ic
forms whose coefficients satisfy finitely many congruence
conditions.

\section{Preliminaries on the Julia invariant} \label{prelim}

In this section, we collect some preliminary facts about the Julia
invariant $\theta$ and the associated quadratic covariant $Q(x,
y)$. The systematic study of these two expressions was begun by Julia
in his thesis~\cite{Julia}, and recently expanded upon by Stoll and
Cremona in~\cite{CS}.

It may not be immediately clear from the definition of either $\theta$ or
$Q(x, y)$ that $\theta$ is an invariant of $f(x, y)$ under the action
of $\SL_{2}(\R)$, but Julia proved this in his thesis~\cite{Julia}.
In fact, this was essentially known to Hermite in the 19th century
(see~\cite[p.~5]{Julia}). Even though the invariant $\theta$, unlike
the rational invariants of classical invariant theory, is not a
polynomial in the coefficients $a_{i}$ of $f(x, y)$, one can still say
that $\theta$ is ``homogeneous of degree 2'' in the following sense:
for any scalar $\lambda\in\R$ and any binary form $f\in V_n(\R)$, we
have $\theta(\lambda f)=\lambda^{2} \theta(f)$.  To see this, notice
that, if $f$ is replaced by $\lambda f$ in (\ref{def:theta}), then
$a_{0}^2$ is multiplied by a factor of $\lambda^2$, while the remaining
factor in this expression remains unchanged; thus $\theta$ gets
multiplied by $\lambda^2$.

As noted earlier, Julia used the definition of $Q(x, y)$ to develop a
theory of reduction for binary $n$-ic forms, which generalizes the theory
defined by Gauss for positive-definite quadratic forms.  Many
beautiful aspects of this theory are discussed by Stoll and Cremona
in~\cite{CS}.  In particular, this reduction theory coincides with the
classical reduction theory for binary cubic forms of positive
discriminant, which uses the Hessian as a quadratic covariant.  The
utility of $\theta$ arises from the fact that Julia showed that, for
reduced binary $n$-ics, one can bound the leading coefficient $a_{0}$
in terms of $\theta$; more precisely, he showed that
\begin{equation}
 a_{0}^{2} \leq \frac{1}{3^{n/2} n^n}\cdot \theta. 
 \end{equation}
Furthermore, Julia showed that one also can bound the magnitude $|\alpha_{i}|$ of the roots of $f(x, y)$ 
in terms of $\theta/a_{0}^{2}$; more precisely, we have
\begin{equation}
 |\alpha_{i}|^{2} \leq \frac{1}{(n-1)^{n-1}3^{n/2}}\cdot \frac{\theta}{a_{0}^{2}} . 
 \end{equation}

Julia provides explicit choices for the parameters $t_{j}$ in the case
of cubic and quartic forms; for the general case he does not give as
many details, but Stoll and Cremona provide a method for determining
the $t_{j}$ (and therefore both $Q(x, y)$ and $\theta(f)$) in the
general case of a binary form of degree $n$.  
  
Because $Q(x, y)$ is a positive-definite quadratic form, there exists
a unique point $z(f)$ in the upper half plane $\HH$ that is a root of
$Q(x, 1)$.  We say that $Q(x, y)$ is {\it reduced} if $z(f)$ lies in
the usual fundamental domain for the action of $\SL_{2}(\Z)$ on $\HH$,
and we say that $f(x,y)$ is ({\it Julia--}){\it reduced} if and only
if $Q(x,y)$ is reduced.

We assume that $f(x, y)$ is what Stoll and Cremona call a
\emph{stable} form: that is, a form that has no repeated roots of
multiplicity $\geq n/2$.  Since we will only be counting irreducible
integral forms, which have no repeated roots, this restriction will
not impact our results.  In~\cite{CS}, Cremona and Stoll prove that,
if we write $z(f) = t + iu$ where $t, u \in
\R$, then
the representative point $z(f)$ of $f$ in the upper half plane $\HH$
is the point $(t, u)$ that minimizes the function
\begin{equation}\label{fdef}
 \frac{\tilde{F}(t, u)}{u^{n}} = \frac{|a_{0}|^{2} \displaystyle \prod_{j=1}^{n} (|t - \alpha_{j}|^{2} + u^{2})}{u^{n}}. 
\end{equation}
The Julia invariant $\theta$ is then the minimal value of this
function, and it is invariant under the action of $\SL_{2}(\R)$.  
For proofs of these assertions, as well as an elegant geometric description and
alternate formulation of this condition using resultants,
see~\cite[Section 5]{CS}.
  
\section{A bounded semialgebraic $\SL_{2}(\R)$-reduced region $L_n$ 
  for real binary  $n$-ics having fixed Julia invariant}
\label{cpt}\label{lip_L}

The objective of this section is to exhibit a fundamental domain $L_n$
for the action of $\SL_{2}(\R)$ on the set of all real binary $n$-ics
having a fixed Julia invariant (say~1), that is semialgebraic and 
lies in a {bounded} set.  (Recall that a set in $V_n(\R)$, which we
identify naturally with ${\R}^{n+1}$, is called
{\it semialgebraic} if it defined by finitely many polynomial
inequalities.)  The construction of $L_n$ will 
be useful to us in defining convenient fundamental domains
for the action of $\SL_2(\Z)$ on real binary $n$-ic forms.

We begin by exhibiting a semialgebraic fundamental domain $E$ for
the action of the usual compact group $K=\SO_2(\R)$ on the whole space
$V_n(\R)$ of real binary $n$-ic forms.  Namely, we define $E$ as the
set of all real binary $n$-ic forms
$f(x,y)=a_0x^n+a_1x^{n-1}y+\cdots+a_ny^{n} \in V_n(\R)$ such that the
associated sequence $S(f)$ given by
$|a_0|,-a_0,|a_1|,-a_1,\ldots,|a_n|,-a_n$ is minimal, with respect to
the lexicographic ordering, among all forms $f'\in K\cdot f$.  Such a
unique form $f$ exists in its orbit $K\cdot f$ because $K$ is compact.
The set $E$ is clearly a fundamental domain for the action of $K$ on
$V_n(\R)$.  Moreover, this set $E\in V_n(\R)$ may evidently be defined
by polynomial equations and inequalities using the logical connectors
$\lor$, $\land$, $\lnot$ and the quantifiers $\forall$, $\exists$,
and hence is semialgebraic by the theorem of Tarski and Seidenberg on
quantifier elimination (see~\cite{T} and \cite{S}).

To construct a bounded semialgebraic fundamental domain $L_n$ for the
action of $\SL_2(\R)$ on real binary $n$-ics having Julia invariant 1,
recall that the representative point $z(f)$ of $f$ in the upper half
plane $\HH$ is the point $t+iu$ that minimizes the function
$\frac{\tilde{F}(t, u)}{u^{n}}$, where $\tilde{F}$ is as defined in
(\ref{fdef}); furthermore, $\theta(f)$ is the minimal value of this
function.  Let $L_n'$ denote the set of all real binary $n$-ic forms
$f(x,y)=a_0x^n+a_1x^{n-1}y+\cdots+a_ny^n$ satisfying $z(f)=i$ and
$\theta(f)=1$.  Then the orthogonal group $K=\SO_2(\R)$, the
stabilizer in $\SL_2(\R)$ of $i\in\HH$, acts on $L_n'$.  Let $L_n$
denote the fundamental domain $L_n'\cap E$ for the action of $K$ on
$L_n'$.

\begin{proposition}\label{funddomain}
The set $L_n$ is a fundamental domain for the action of $\SL_2(\R)$ on the
set of real binary $n$-ic forms having Julia invariant~$1$ and, moreover,
$L_n$ is bounded and semialgebraic.
\end{proposition}

\begin{proof}
\textbf{$L_n$ is a fundamental domain.}  Let $f$ be any real binary
$n$-ic form having Julia invariant 1.  Then there exists an element
$\gamma\in\SL_{2}(\R)$ that sends the representative point $z(f)$ to
$i$, because $\SL_{2}(\R)$ acts transitively on the upper half plane.
Furthermore, since $z(f)$ is a covariant of $f$, if we act on $f$ by
this same element $\gamma$, the resulting binary $n$-ic form will have
$z(f)=i$ as its representative point in the upper half plane.  In
addition, $\gamma$ is uniquely determined up to left multiplication by
elements of $K$, the stabilizer in $\SL_2(\R)$ of $i\in\HH$.  Thus, for
any real binary $n$-ic form $f$ with Julia invariant~1, by the
definition of $L_n$ there exists a unique associated element
$\gamma\cdot f$ $(\gamma\in\SL_2(\R))$ such that $\gamma\cdot f\in
L_n$; hence $L_n$ is a fundamental domain for the action of
$\SL_2(\R)$ on real binary $n$-ic forms having Julia invariant~1.

\vspace{.1in}
\noindent \textbf{$L_n$ is bounded.} It suffices to show that $L_n'$
lies in a bounded subset of $\R^{n+1}$.  Suppose that $f\in L_n'$,
i.e., $f$ is a form with $z(f) = i$ and $\theta(f) = 1$.  Then
\begin{equation}\label{prod1}
 \theta(f)=\frac{\tilde{F}(t,u)}{u^n} = |a_{0}|^2 \prod_{j = 1}^{n} (|\alpha_{j}|^2 + 1) = 1, \end{equation}
which is obtained by setting $t = 0$ and $u = 1$ in (\ref{fdef}).  In particular, this implies that
\begin{equation}	\label{prod}
 \prod_{j=1}^n (|\alpha_{j}|^2 + 1) = \frac{1}{a_{0}^{2}}. 
\end{equation}
If we expand the product in the expression on the left hand side 
of (\ref{prod}), we see that the square of the absolute value of 
each (distinct) $k$-fold product of the $n$ roots of $f$ appears, 
for every $k\in\{0,\ldots,n\}$.  Since each of the terms appearing
in this expanded product is nonnegative, each is then bounded by 
$1/a_{0}^{2}$.  For example, in the case $k = 1$, note that each 
$|\alpha_{i}|^{2}$ appears in this product, and so we have a bound 
of the form

\[ |\alpha_{1}|^{2} + \ldots + |\alpha_{n}|^{2} \leq \frac{1}{a_{0}^{2}}. \]

Since $a_{k}/a_{0}$ is, up to sign, the sum of the distinct $k$-fold 
products of the roots $\alpha_i$ of $f(x,1)$, by the Cauchy--Schwartz inequality 
we obtain
\[ \left| \frac{a_{k}}{a_{0}} \right| = \left|\sum_{1\leq
  i_1<\cdots<i_k\leq n} \alpha_{i_{1}} \cdots \alpha_{i_{k}} \right|
\leq \left({n\choose k} \sum_{1\leq i_1<\cdots<i_k\leq
    n}\left|\alpha_{i_{1}} \cdots
    \alpha_{i_{k}}\right|^{2}\right)^{1/2} \leq  \frac{1}{|a_0|}{n\choose k}^{1/2}, \]
which implies that $|a_{k}| \leq {n\choose k}^{1/2}<2^{n/2}$.
This shows that the forms $f$ in $L_n'$ have the property that 
all coefficients are less than $2^{n/2}$ in absolute value; 
thus the set $L_n'$ (and hence $L_n$) is indeed contained in a bounded set.

\vspace{.1in}
\noindent \textbf{$L_n$ is semialgebraic.} Again, it suffices to show that the
set $L_n'$ is semialgebraic. By~\cite[Equations 4.5]{CS}, the condition
that $z(f) = i$ is equivalent to the condition that the roots
$\alpha_1, \ldots, \alpha_n$ of $f(x, 1)$ satisfy the two equations
\begin{eqnarray}
\sum_{j=1}^{n} \frac{1}{|\alpha_j|^2 + 1} & = & \frac{n}{2},  \label{firsteq}\\
\sum_{j=1}^{n} \frac{-\alpha_j}{|\alpha_j|^2 + 1} & = & 0 \label{secondeq}\,.
\end{eqnarray}
In addition, when $z(f)=i$, by equation (\ref{prod1}) the condition that $\theta(f) = 1$ 
is equivalent to 
\begin{equation} |a_0|^2 \prod_{j} (|\alpha_j|^2 + 1) = 1. 
\end{equation}
These three equations taken together define a semialgebraic set in the
space whose coordinates are $(a_0, \alpha_1, \ldots, \alpha_n)$.  (It
is possible that some of the $\alpha_i$ are complex, in which case we think of
each such $\alpha_i$ as an element of $\R^2$.)  Since
there is a polynomial map from the space with coordinates $(a_0,
\alpha_1, \ldots, \alpha_n)$ to the space of coefficients $(a_0,
\ldots, a_n)$ of $f(x, y)$ (namely, the polynomial map which expresses
each coefficient $a_i$ as a function of $a_0$ and the $\alpha_i$), and
polynomial images of semialgebraic sets are semialgebraic by the
theorem of Tarski and Seidenberg, this
shows that $L_n'$ is also semialgebraic.  The set $L_n=L_n'\cap E$,
being the intersection of two semialgebraic sets, is then also
semialgebraic. 
\end{proof}

For each fixed $n$, we have thus obtained a fundamental domain
$L_n$, for the action of $\SL_2(\R)$ on binary $n$-ic forms having
Julia invariant~1, that is bounded and is
defined by some fixed set of polynomial equalities and inequalities.
More generally, by restricting the above construction to just those
real binary $n$-ic forms that have $n-2k$ real roots (which is also a
semialgebraic subset of $V_n(\R)\cong\R^{n+1}$), we obtain a
fundamental domain $L_{n,k}\subset L_n$ for the action of $\SL_2(\R)$
on real binary $n$-ic forms having $n-2k$ real roots and Julia
invariant~1, which is again bounded and semialgebraic.

\section{Reduction theory for the action of $\SL_{2}(\Z)$ on binary $n$-ics}\label{redthbq}

Let $k\in\{0,1,\ldots,\lfloor n/2\rfloor\}$, and let $L_{n,k}$ denote
a fundamental domain for the action of $\SL_2(\R)$ on the open subset
$V_{n,k}\subset V_n(\R)$ of those nondegenerate
binary $n$-ic forms $f$ with coefficients in $\R$ having $n-2k$ real
roots, and satisfying $\theta(f) = 1$; here, a binary $n$-ic form is
called nondegenerate if it has nonzero discriminant.
By the previous section, we may
assume that $L_{n,k}$ is bounded and semialgebraic.  For convenience,
we will assume for now (until Remark~\ref{quadraticremark}) that $n\geq 3$.

Let $\FF$ denote Gauss's usual fundamental domain for
$\GL_2^+(\Z)$ acting on $\GL_2^+(\R)$, where $\GL_2^+(\R)$ is the subgroup of $\GL_2(\R)$ of elements having positive determinant, and $\GL_2^+(\Z)$ is simply $\GL_2^+(\R)\cap \GL_2(\Z)=\SL_2(\Z)$.  Then $\FF$ may be expressed in the form $\FF=
\{\nu\alpha \kappa \lambda:\nu = \nu(u)\in N'(\alpha),\alpha=\alpha(t)\in A',\kappa\in K,\lambda\in\Lambda\}$, where
\begin{equation}\label{nak}
\begin{array}{rcl}
N'(\alpha)&=& \left\{\nu(u)=\left(\begin{array}{cc} 1 & {} \\ {u} & 1 \end{array}\right)\;:\;
        u\in\mathcal I(\alpha) \right\}  \;  , \\[.25in]
A' &=& \left\{\alpha(t)=\left(\begin{array}{cc} t^{-1} & {} \\ {} & t \end{array}\right)\;:\;
       t^2\geq \sqrt3/2 \right\} \; , \\[.25in]
\Lambda &=& \left\{\lambda=\left(\begin{array}{cc} \lambda & {} \\ {} & \lambda 
        \end{array}\right)\;:\;
        \lambda>0 \right\}\;,
\end{array}
\end{equation}
and $K$ is the usual (compact) real orthogonal group ${\rm SO}_2(\R)$;
here $\mathcal I(\alpha)$ is a union of one or two subintervals of
$[-\frac12,\frac12]$ depending only on the value of $\alpha\in A'$.
We use $N\subset\SL_2(\R)$ to denote the subgroup of all matrices of
the form $\nu(u)$ ($u\in \R$) and $A\subset \SL_2(\R)$ to denote the
subgroup of all diagonal matrices $\alpha(t)$ ($t\in\R^\times$) of
determinant 1, so that $N'\subset N$ and $A'\subset A$.  In this
notation, we also have the Iwasawa decomposition $\SL_2(\R)=NAK$.

Let $m=m(n,k)$ denote the size of  
$\#\Stab_{\SL_2(\R)}(v)/\#\Stab_{\SL_2(\Z)}(v)$ 
for a generic element $v\in V_{n,k}$ (i.e., for $v$ outside a set of measure 0 in $V_{n,k}$).  Then it is easy to see and well-known that $m=3$ if
$(n,k)=(3,0)$; $m=4$ if $(n,k)=(4,0)$ or $(4,2)$; $m=2$ if
$(n,k)=(4,1)$; and $m=1$ otherwise.  

Let $L:=L_{n,k}$.  For $h\in \GL_2(\R)$, we regard $\FF hL$ as a
multiset, where the multiplicity of a point $v$ in $\FF h L$ is the
cardinality of the set $\{g\in \FF: v\in ghL\}$.  By the argument of
\cite[\S2.1]{BS}, the $\SL_2(\Z)$-equivalence class of $v\in V_{n,k}$
is represented $m_v=\#\Stab_{\SL_2(\R)}(v)/\#\Stab_{\SL_2(\Z)}(v)$
times in $\FF h L$.  It follows, as in \cite[\S2.1]{BS}, that away from
a measure zero set (where $m_v\neq m$), the multiset $\FF hL$ is the union of $m$
fundamental domains for the action of $\SL_2(\Z)$ on $V_{n,k}$.  

Thus for any $h\in \GL_2(\R)$, if we let $\RR_X(hL)$ denote the
multiset $\{w\in\FF h L:\theta(w)<X\}$, then the product $mN_{n,k}(X)$ 
is equal to the number of
irreducible integer points in $\RR_X(hL)$, 
with the slight caveat that
the (relatively rare--see Corollary~\ref{mcor}) 
integer points $v\in V_{n,k}$ with 
$m_v\neq m$ are counted with weight $m/m_v$.

Thus, to determine the asymptotic behavior of $N_{n,k}(X)$, it suffices to
count the number of lattice points in $\RR_X(h L)$.  
However, one major obstacle to counting integer points
of bounded height in $\RR_X(hL)$ is that it is not bounded,
but rather has a cusp going off to infinity.  We simplify the counting
in this cuspidal region by ``thickening'' the cusp; more precisely, we
compute the number of integer points in $\RR_X( hL)$ by averaging over
a compact continuum of such fundamental domains, where $h$ ranges over some
suitable compact subset $G_0\subset \GL_2(\R)$.  This adaptation of the
method of \cite{dodpf} is described in more detail in
\S\ref{avgsec}.

However, in \S\ref{redsec} we first examine the problem of estimating the
number of reducible points in the main bodies (i.e., away from the
cusps) of our fundamental domains.

\section{Estimates on reducibility}\label{redsec}

We first consider the integral elements in the region
$\RR_X(hL):=\{f\in \FF hL:\theta(f)<X\}$ that are
reducible over $\Q$, where $h$ is any element in a fixed compact
subset $G_0$ of $\GL_2(\R)$.  Let $V_n(\Z)$ denote the lattice of integral binary $n$-ic forms in 
$V_n(\R)$.  Note that if a binary $n$-ic form
$a_0x^n+a_1x^{n-1}y+\cdots+a_ny^n\in V_n(\Z)$ satisfies $a_0=0$, then it is
automatically reducible over $\Q$, since $y$ is a factor.  The
following lemma shows that for integral binary $n$-ic forms in
$\RR_X(hL)$, reducibility with $a_0\neq0$ does not occur very
often:

\begin{lemma}\label{reducible}
  Let $h\in G_0$ be any element, where $G_0$ is any fixed compact
  subset of ${\GL}_2(\R)$.  Then the number of integral binary $n$-ic
  forms $a_0x^n+a_1x^{n-1}y+\cdots+a_ny^n  \in \RR_X(hL)$ 
  that are  reducible over $\Q$ with $a_0\neq 0$ is $O(X^{\frac{n+1}{2} -\frac12+ \varepsilon})$, where
  the implied constant depends only on~$n$, $G_0$, and $\varepsilon$.
\end{lemma}

\begin{proof}
Let $f(x,y)=a_0x^n+a_1x^{n-1}y+\cdots+a_ny^n$ be any element in 
$\RR_X(hL)\cap V_n(\Z)$ with $a_0\neq 0$.  Since coefficients of forms in $hL$ are
uniformly bounded, and since 
$\RR_X(hL)\subset N'A'K\Lambda
hL$ (where $0 < \lambda \ll X^{1/2n}$, with the absolute constant only depending on $G_{0}$), we see that 
$$\prod_{{\scriptstyle 0\leq i\leq n-1}\atop{\scriptstyle a_i= 0}}a_0   \prod_{\scriptstyle {0\leq i\leq n-1}\atop{\scriptstyle a_i\neq 0}}a_i = O(X^{\textstyle \frac{n}{2}}),$$  
implying that that the number of points in
$\RR_X(hL)$ with $a_0\neq 0$ and $a_n=0$ is $O(X^{\frac{n}{2}+\varepsilon})$.
Indeed, the actions of $N$ and $K$ only change coefficients by an absolute constant, while
a generic element of $A'$ sends the coefficients 
$(a_{0}, \ldots, a_{n})$ to $(a_{0} a^{-n}, a_{1}a^{-(n-2)}, \ldots, a_{n} a^{n})$;
the bound above follows (recall that we chose $A$ such that $a$ is bounded from below).
Hence we may assume that $a_0 \neq 0$ and $a_n \neq 0$.  

Now suppose that $f$ factors as $f=rs$, where
$r,s$ are binary forms where $r$ has degree $k\geq 1$ and $s$ has degree $n-k$, such that $k\leq n-k$.  
We write $r(x,y)=b_0x^k+b_1x^{k-1}y+\cdots+b_ky^k$ 
and $s(x,y)=c_0x^{n-k}+c_1x^{n-k-1}y+\cdots+c_{n-k}y^{n-k}$.  Then the assumption that $a_0, a_n \neq 0$
 implies that we also must have $b_0,c_0 \neq 0$.

Since $f\in \RR_X(hL)$, we may write $f=\nu\alpha\kappa \lambda h f_0$, where
$\nu\in N'(\alpha)$, $\alpha\in A'$, $\kappa\in K$, $\lambda\in\R_{>0}$ with
$\lambda=O(X^{1/2n})$, and $f_0\in L$.  If we define the height $H(F)$
of a binary form $F$ as the maximum of the absolute values of its
coefficients, since $L$ lies in a compact set, we have $H(f_0) \ll 1$.
Furthermore, the factorization of $f$ as $f=rs$ corresponds to a
factorization $f_0=r_0s_0$, so that just as $f=\nu\alpha \kappa \lambda
hf_0$, we also have $r=\nu\alpha\kappa \lambda hr_0$ and $s=\nu\alpha 
\kappa\lambda hs_0$, where $r_0$ and $s_0$ are real polynomials of degree
$k$ and $n-k$, respectively.  By Gelfond's inequality
(see~\cite[Theorem~4.2.2]{Prasolov}), since $f_0=r_0s_0$, we have

\begin{equation}	\label{gelineq}
	H(r_0)H(s_0)\leq 2^{n-2}\sqrt{n+1}H(f_0)=O(1).
\end{equation}  
Since $\nu$ acts by a bounded lower triangular transformation, $\alpha$ acts by $\bigl(\begin{smallmatrix}t^{-1}& \\ & t\end{smallmatrix}\bigr)$ for some $t\gg1$, $K$ is compact, and $\lambda=O(X^{1/2n})$, it follows from (\ref{gelineq}) that
$$\Bigl[\prod_{{\scriptstyle0\leq i\leq k}\atop{\scriptstyle b_i= 0}}|b_0|   \prod_{{\scriptstyle 0\leq i\leq k}\atop{\scriptstyle b_i\neq 0}}|b_i|\Bigr]^{\textstyle\frac1{k+1}}\cdot 
\Bigl[\prod_{\scriptstyle{0\leq j\leq n-k}\atop{\scriptstyle c_j= 0}}|c_0|   \prod_{{{\scriptstyle0\leq j\leq n-k}\atop{\scriptstyle c_j\neq 0}}}|c_j|\Bigr]^{\textstyle\frac{1}{n-k+1}} 
= O(X^{\textstyle\frac12}),$$  
or equivalently,

\begin{equation}	\label{subject}
\Bigl[\prod_{{\scriptstyle 0\leq i\leq k}\atop{\scriptstyle b_i= 0}}|b_0|   \prod_{{\scriptstyle0\leq i\leq k}\atop{\scriptstyle b_i\neq 0}}|b_i|\Bigr]^{\textstyle\frac{n-k+1}{k+1}}\cdot 
\Bigl[\prod_{{\scriptstyle 0\leq j\leq n-k}\atop{\scriptstyle c_j= 0}}|c_0|   \prod_{{{\scriptstyle 0\leq j\leq n-k}\atop{\scriptstyle c_j\neq 0}}}|c_j|\Bigr] 
= O(X^{\textstyle\frac{n-k+1}2}).
\end{equation}  

The number of integer possibilities for the $b_i$ and $c_j$, subject to (\ref{subject}), is evidently at most 
$O(X^{\frac{n-k+1}2+\varepsilon})$.  Since by assumption $k\geq 1$ (i.e., $f$ factors nontrivially), we obtain the desired estimate.
\end{proof}

In fact, we may prove the stronger statement that 
most (i.e., 100\%) of binary $n$-ic
forms in the fundamental domain $\RR_X(hL)$ with $a_{0} \neq 0$ are not only irreducible
but also have associated
Galois group $S_n$. For monic polynomials ordered by the maximum of
the absolute values of their coefficients, this is a well-known result of
van der Waerden~\cite{vdW}.
Specifically, van der Waerden showed that among the $\sim (2H)^n$ monic
integral polynomials of degree~$n$ whose coefficients are bounded in absolute value by $H$, at
most $O(H^{n-6/((n-2)\log\log n)})$ have associated Galois group not $S_n$.  This was subsequently 
improved by Gallagher~\cite{Gal} to $O(H^{n-1/2}\log H)$, by Zywina~\cite{Zy} to $O(H^{n-1/2})$, by  Dietmann~\cite{Dietmann2} to $O(H^{n-2+\sqrt{2}+\varepsilon})$, by 
Anderson, Gafni, Lemke Oliver, Lowry-Duda, Shakan, and Zhang~\cite{aimgroup} 
to $O(H^{n-\frac23+\frac2{3n+3}+\varepsilon})$, and most recently to the optimal $O(H^{n-1})$ in \cite{Bh}.

These results do not directly apply to the situation at hand, as we
are counting polynomials in a noncompact fundamental domain for
$\SL_2(\Z)$ rather than in a compact box having equal-length sides.
Nevertheless, the methods of Dietmann~\cite{Dietmann} can be adapted to our
situation to yield the following:

\begin{theorem}\label{nonSn}
  Let $h\in G_0$ be any element, where $G_0$ is any fixed compact
  subset of ${\GL}_2(\R)$.  Then the number of integral binary $n$-ic
  forms $a_0x^n+a_1x^{n-1}y+\cdots+a_ny^n  \in \RR_X(hL)$ 
  with $a_0\neq 0$ 
  whose Galois group over $\Q$ is not isomorphic to $S_n$
    is $O(X^{{\frac{n+1}{2}-\frac14} + \varepsilon})$, where
  the implied constant depends only on~$n$, $G_0$, and $\varepsilon$.
\end{theorem}

\begin{proof}
While the methods of either \cite{Cohen} or \cite{Dietmann} can be
  adapted to prove this result, we use the methods of
  \cite{Dietmann} as they are technically simpler.  

  First, we note that the ideas of~\cite{Dietmann} can be applied
  even to integral polynomials $g(x)=a_0x^n+a_1x^{n-1}+\cdots+a_n$
  that are not necessarily monic, i.e., for which $a_0\neq 1$, so long
  as $a_0$ is nonzero.  The
  reason is that~\cite[Lemma~2]{Dietmann} holds also for such nonmonic
  polynomials $g(x)$: simply apply the proof there to $h(x)=a_0^{n-1}g(x/a_0)$,
  which is monic, and then the identical result is then seen to hold true
  for $g(x)$.  The definition of resolvent in \cite[Lemma~5]{Dietmann} can also
  be modified similarly, again by replacing the resolvent $r(x)$ as given by 
  $a_0^{\deg(r)-1}r(x/a_0)$, so that the modified resolvent is again
  integral.  All arguments then apply in the identical manner.

  To obtain Theorem~\ref{nonSn}, we now proceed as follows.  Suppose we are given
  the coefficients $a_0,\ldots,a_{n-2}\in\Z$ of
  $f(x,y;a_{n-1},a_n)=a_0x^n+a_1x^{n-1}y+\cdots+a_ny^n$, where $a_0\neq 0$.  Then 
  \cite[Lemma~2]{Dietmann}, as modified above, implies that
  there are only at most $n^2+n$ integral values of $a_{n-1}$ such
  that $f$ does not have associated Galois group $S_n$ over $\Q(a_n)$.  
  Since 
  $$
\prod_{{\scriptstyle{i\neq n-1}}\atop {\scriptstyle a_i=0}} |a_0| \prod_{{\scriptstyle i\neq n-1}\atop{\scriptstyle a_i\neq0}}|a_i| = O(X^{\textstyle \frac{n+1}2\!-\!\frac{1}{2}})$$
  for integral binary forms $f(x,y)$ in $\RR_X(hL)$,
  the total number of such binary $n$-ic forms  with $a_0\neq 0$ is at most $O(X^{\frac{n+1}2-\frac{1}{2}+\varepsilon})$.
   
  Next, suppose again that $a_0,\ldots,a_{n-2}$ are given, and furthermore suppose that $a_{n-1}\in\Z$ is not among the above $n^2+n$ distinguished values, so that the associated Galois group of the binary form $f$ over $\Q(a_n)$ is in fact $S_n$.  Since
$$ a_n = O\bigl(X^{\textstyle\frac{n+1}{2}}/\mathcal A  \bigr)$$
where
$$ \mathcal A = \prod_{{\scriptstyle1\leq i\leq n-1}\atop {\scriptstyle a_i=0}} |a_0| \prod_{{\scriptstyle 0\leq i\leq n-1}\atop{\scriptstyle a_i\neq0}}|a_i|,$$
the argument of \cite[Proof of Theorem 1]{Dietmann} shows
that at most $O\bigl((X^{\frac{n+1}{2}}/\mathcal A)^{1/2} \bigr)$ of
these values of $a_n$ can yield binary forms $f(x,y)$ having
associated Galois group smaller than $S_n$.  Since the number of values of $a_0,\ldots,a_{n-1}$ yielding a given value of $\mathcal A$ is $O(\mathcal A^{\varepsilon})$, the number of possible values of $a_0,\ldots,a_n$ is thus at most
$$\sum_{\mathcal A=O(X^{\frac{n}{2}})}
   O\bigl(\mathcal A^{\varepsilon}(X^{\textstyle \frac{n+1}{2}}/\mathcal A)^{1/2} \bigr)=
   O(X^{{\textstyle \frac{n+1}{2}}-{\textstyle\frac14}+\varepsilon}),$$ 
yielding the desired result.
\end{proof}

\noindent
The $O$-estimate in Theorem~\ref{nonSn} can be further improved to $O(X^{{\frac{n+1}{2}-\frac12}})$
 using the methods of~\cite{Bh}, although we shall not require this improvement here when $n>3$.  
For $n=3$, the further improved estimate $O(X)$ 
can be deduced from \cite{BhSh} (see also~\cite[Lemma~2]{Davenport} for a proof of the estimate
 $O(X^{\frac32+\varepsilon})$), while for $n=2$, we observe that every definite integral binary quadratic
 form has Galois group $S_2$. 

One interesting and useful consequence of a binary $n$-ic form $f$ having
associated Galois group $S_n$ ($n\geq 5$) is that  in that case $f$ cannot have any
nontrivial projective linear automorphisms over $\bar\Q$, i.e., there cannot exist elements in $\SL_2(\bar\Q)$ that stabilize $f$ and induce a nontrivial permutation of the roots of $f$:

\begin{theorem}\label{Sntrivstab}
Suppose $n\geq5$.  If a binary $n$-ic form $f(x,y)\in V_n(\Z)$ is irreducible with Galois group $S_n$, 
then $f$ has no projective linear automorphisms over $\bar\Q$.
\end{theorem}

\begin{proof}
Suppose $n\geq 5$. Let $f$ be an integral binary $n$-ic form having associated Galois group $S_n\cong G\subset \Gal(\bar \Q/\Q)$.  Let $H\subset G$ denote the subgroup of those symmetries of the roots of $f$ in $\P^1(\bar\Q)$ that come from symmetries of $\P^1(\bar\Q)$ in $\PGL_2(\bar\Q)$.  Then $H$ is normal in $G$; indeed, if $h\in H$ and $g\in G$, then $ghg^{-1}$ is again in $H$, for if we write $h(x) = \frac{ax+b}{cx+d}$, then 
\[ghg^{-1}(x) = g\Bigl(\frac{ag^{-1}(x)+b}{cg^{-1}(x)+d}\Bigr)=\frac{g(a)x+g(b)}{g(c)x+g(d)}.\]
It follows from a result of Olver~\cite[Corollary~8.68]{Olver} that for $n\geq 5$, we have $|H|\leq 4n-8$.  However, for $n\geq 5$, the only subgroup of $S_n$ that is normal and of cardinality at most $4n-8$ is the trivial subgroup, and Theorem~\ref{Sntrivstab} follows.\end{proof}

\begin{corollary}\label{mcor}
  Let $h\in G_0$ be any element, where $G_0$ is any fixed compact
  subset of ${\SL}_2(\R)$.  Then all but $O(X^{{\frac{n+1}{2}-\frac14} + \varepsilon})$ of the integral binary $n$-ic  forms $f(x,y)=a_0x^n+a_1x^{n-1}y+\cdots+a_ny^n  \in \RR_X(hL)$ 
  with $a_0\neq 0$ are irreducible over $\Q$, have associated Galois group $S_n$, and satisfy $m_f=m$. $($Here again the implied constant depends only on~$n$, $G_0$, and $\varepsilon$.$)$
\end{corollary}

\begin{proof}
In the case of $n=3$, this follows directly from Theorem~\ref{nonSn} and
\cite[Lemma 2]{Davenport}, while in the case $n=4$, the argument is
identical to \cite[Proofs of Lemmas~2.2 and 2.4]{BS}.
For~$n=5$, the assertion follows from 
Theorems~\ref{nonSn} and \ref{Sntrivstab}; indeed, the stabilizer in $\GL_2(\C)$ of a binary $n$-ic
form $f$ is an extension of the projective automorphism group of $f$ over $\C$ by the group of $n$th roots
of unity in $\C^\times$, and the only $n$th roots of unity in $\R^\times$ (or $\Z^\times$) are 1 or $\pm1$ depending on whether $n$ is odd or even.  This completes the proof.
\end{proof}

\section{Averaging and cutting off the cusp}\label{avgsec}

Let $G_0$ be a compact, semialgebraic, left $K$-invariant set in $\GL_2(\R)$ that is the closure
of a nonempty open set and in which every element has determinant
greater than or equal to~$1$.
Then we may write

\begin{equation}
N_{n,k}(X)=\frac{\int_{h\in G_0}\#\{x\in \FF hL\cap V_n(\Z)^{\irr}:
  \theta(x)<X\}dh\;}{m\cdot\int_{h\in G_0}dh},
\end{equation}
where $L:=L_{n,k}$, $V_n(\Z)^\irr$ denotes the set of irreducible elements in
$V_n(\Z)$, and $dh$ is Haar measure on $\GL_2(\R)$.  The denominator of the
latter expression is an absolute constant $C_{G_0}^{n,k}$ greater than
zero.

More generally, for any $\SL_2(\Z)$-invariant subset $S \subset V_{n,k}(\Z):=
V_n(\Z)\cap V_{n,k}$, let $N(S;X)$ denote the number of irreducible
$\SL_2(\Z)$-orbits in $S$ having Julia invariant less than $X$. Let $S^{\irr}$
denote the subset of irreducible points of $S$. Then $N(S;X)$ can be
similarly expressed as

\begin{equation}\label{eq9}
N(S;X)=\frac{\int_{h\in G_0}\#\{x\in \FF hL\cap S^{\irr}: \theta(x)<X\}dh\;}{C_{G_0}^{n,k}}.
\end{equation}

Now, given $x\in V_{n,k}$, let $x_L$ denote the {unique} point in $L$
that is equivalent by an element of $\GL_2^+(\R)$ to $x$. Then 

\begin{equation}
N(S;X)=\frac{1}{C_{G_0}^{n,k}}\sum_{\substack{{x\in
      S^{\irr}}\\[.02in]{\theta(x)<X}}}\int_{h\in G_0} \#\{g \in \FF :
x=ghx_L\} dh.
\end{equation}
For a given $x\in S^{\irr}$, since $n\geq 3$, there exist a finite number of elements
$g_1,\ldots,g_r\in\GL_2^+(\R)$ satisfying $g_jx_L=x$.  We then have
$$\int_{h\in G_0} \#\{g \in \FF :x=ghx_L\} dh=\sum_j\int_{h\in G_0} \#\{g \in \FF :gh=g_j\} dh=\sum_j\int_{h\in G_0\cap\FF^{-1}g_j}dh.$$
As $dh$ is an invariant measure on $\GL_2(\R)$, we have
\begin{eqnarray*}
\sum_j\displaystyle\int_{h\in G_0\cap\FF^{-1}g_j}\!\!\!\!\!dh=\sum_j\displaystyle\int_{g\in G_0g_j^{-1}\cap\FF^{-1}}\!\!\!\!\!dg&=&\sum_j\displaystyle\int_{g\in\FF}\#\{h \in G_0 :gh=g_j\} dg \\ &=&\int_{g\in \FF} \#\{h \in G_0 :x=ghx_L\} dg.
\end{eqnarray*}
Therefore,
\begin{eqnarray*}\label{eqavg}
N(S;X)&=&\frac{1}{C_{G_0}^{n,k}}\,\sum_{\substack{x\in S^{\irr}\\[.02in]
    \theta(x)<X}} \int_{g\in\FF} \#\{h \in G_0 : x=ghx_L\}dg\\
&\!\!=\!\!&  \frac1{C_{G_0}^{n,k}}\int_{g\in\FF}
\#\{x\in S^\irr\cap gG_0L:\theta(x)<X\}dg\\[.075in] 
&\!\!=\!\!&  \frac1{C_{G_0}^{n,k}}\int_{g\in N'(t)A'\Lambda K}
\#\{x\in S^\irr\cap n \bigl(\begin{smallmatrix}t^{-1}& {}\\
  {}& t\end{smallmatrix}\bigr) \lambda \kappa G_0L:\theta(x)<X\}t^{-2} 
dn\, d^\times t\,d^\times \lambda\, d\kappa\,.
\end{eqnarray*}
Let us write $B(u,t,\lambda,X) = \nu(u) \bigl(\begin{smallmatrix}t^{-1}& {}\\ 
{}& t\end{smallmatrix}\bigr) \lambda G_0L\cap\{x\in V_{n,k}:\theta(x)<X\}$.
Then since $G_0$ is left $K$-invariant, and we may normalize Haar measure so that $\int_{\kappa\in K} d\kappa=1$, we obtain
\begin{equation}\label{avg}
N(S;X) = \frac1{C_{G_0}^{n,k}}\int_{g\in N'(t)A'\Lambda}                              
\#\{x\in S^\irr\cap B(u,t,\lambda,X)\}t^{-2}
du\, d^\times t\,d^\times \lambda\,.
\end{equation}

To estimate the number of lattice points in $B(u,t,\lambda,X)$, we
have the following proposition due to Davenport~\cite{dav:lip}. 

\begin{proposition}\label{davlem}
  Let $\mathcal R$ be a bounded, semialgebraic multiset in $\R^n$
  having maximum multiplicity $m$, and that is defined by at most $\kappa$
  polynomial inequalities each having degree at most $\ell$.  Let $\RR'$
  denote the image of $\RR$ under any $($upper or lower$)$ triangular,
  unipotent transformation of $\R^n$.  Then the number of integer
  lattice points $($counted with multiplicity$)$ contained in the
  region $\mathcal R'$ is
\[\Vol(\mathcal R)+ O(\max\{\Vol(\bar{\mathcal R}),1\}),\]
where $\Vol(\bar{\mathcal R})$ denotes the greatest $d$-dimensional 
volume of any projection of $\mathcal R$ onto a coordinate subspace
obtained by equating $n-d$ coordinates to zero, where 
$d$ takes all values from
$1$ to $n-1$.  The implied constant in the second summand depends
only on $n$, $m$, $\kappa$, and $\ell$.
\end{proposition}
Although Davenport states the above lemma only for compact
semialgebraic sets $\mathcal R\subset\R^n$, his proof adapts without
significant change to the more general case of a bounded semialgebraic
multiset $\mathcal R\subset\R^n$, with the same estimate applying also to
any image $\mathcal R'$ of $\mathcal R$ under a unipotent triangular
transformation.

By our construction of $L$, the coefficients of the binary
$n$-ic forms in $G_0L$ are all uniformly bounded.
Let $C^n$ be a constant that bounds the absolute value of the leading
coefficient $a_0$ of all the forms $a_0x^n+a_1x^{n-1}y+\cdots+a_ny^n$ in
$G_0L$.  (We choose $C^n$ instead of $C$ to simplify the exponents of $C$
in the calculations which follow.)  

We then have the following lemma on the number of irreducible lattice
points in $B(u,t,\lambda,X)$: 

\begin{proposition}\label{aneq0}
The number of lattice points $(a_0,\ldots,a_n)$ in $B(u,t,\lambda,X)$ with $a_0\neq0$ is
$$\left\{
\begin{array}{cl}
0 & \mbox{{\em if} $\frac{C\lambda}{t}<1$};\\[.1in]
\Vol(B(u,t,\lambda,X)) + O(\max\{ t^n \lambda^{n^2},1\}) &\mbox{{\em otherwise,}}
\end{array}\right.$$
where the implied constant in the big-$O$ expression depends only
on $n$ and $G_{0}$.
\end{proposition}

\begin{proof}
  If ${C {\lambda/ t}<1}$, then $a_{0}=0$ is the only possibility for an
  integral binary $n$-ic form $a_{0}x^n+\ldots + a_{n}y^n$ in
  $B(u,t,\lambda,X)$, and any such form is reducible.  Indeed, notice that 
  any element in $B(u, t, \lambda, X)$ has first coordinate bounded by 
  $C^n \lambda^n / t^n$, which is $< 1$ if $C\lambda/t < 1$.  If
  ${C {\lambda/ t}\geq1}$, then $\lambda$ and $t$ are positive numbers
  bounded from below by $(\sqrt3/2)^{1/2}/C$ and $(\sqrt3/2)^{1/2}$ respectively.
  In this case, one sees that the projection
  of $B(u,t,\lambda,X)$ onto $a_{0}=0$ has volume $O(t^n\lambda^{n^2})$:
  the coefficients of forms in $G_{0} L$ are uniformly bounded, and acting by 
  the scalar $\lambda$ scales each coefficient by a factor of $\lambda^n$.  Acting
  by the scalar matrix $\bigl(\begin{smallmatrix}t^{-1}& {}\\ 
{}& t\end{smallmatrix}\bigr)$ then multiplies the $k$th coefficient by $t^{-n + 2k}$.  
  Thus, after acting by these two elements, $a_{k}$ is multiplied by a factor of
  $\lambda^{n} t^{-n + 2k}$.  The product of these numbers, for $0 \leq k \leq n$, 
  is $\lambda^{n(n+1)} t^{0}$, and represents a big-$O$ upper bound for the volume of
  $B(0, t, \lambda, X)$ and therefore also $B(u,t,\lambda,X)$.  Therefore, if we project this region onto $a_{0} = 0$, an upper bound for the
  volume of this projection is given by $\lambda^{n(n+1)} t^{0} / (\lambda^{n} t^{-n}) = \lambda^{n^2} t^{n}$,
  as claimed.  
      
  Now consider any other projection of $B(u, t, \lambda, X)$ onto one of the subspaces $a_{k} = 0$, say.
  The volume of this projection is given by $O(\lambda^{n^2} t^{n - 2k})$.  This is evidently $O(\lambda^{n^2} t^{n})$, since $t$ is uniformly bounded from below.  If we want to project onto a space defined by an additional 
  condition $a_{\ell} = 0$, say, this space will have volume bounded by 
  $O(\lambda^{n^2} t^{n} / (\lambda^{n} t^{-n + 2\ell})) = O(\lambda^{n^2 - n} t^{2n - 2\ell})$.  
  Although the exponent of $t$ might increase, we use the fact that $C\lambda > t$ to exchange $n - 2\ell$ 
  (which is $\leq n$) factors of $t$ for $n - 2\ell \leq n$ factors of $\lambda$, which shows that this expression
  is still $O(\lambda^{n^2} t^{n})$.  It is clear that we may interchange powers of $t$ for powers of $\lambda$
  in any projection of the original region onto any proper subspace spanned by coordinate axes to get an upper
  bound of $O(\lambda^{n^2} t^{n})$ on their volumes.
  The lemma then follows from Proposition \ref{davlem}.
\end{proof}

In $(\ref{avg})$, since $L$ (and therefore also $G_0L$) only
contains points with Julia invariant at least $1$, we observe (by the definition
of $B(u,t,\lambda,X)$) that the integrand 
will be nonzero only if $t<C\lambda$ and $\lambda<X^{1/2n}$.  Thus we
may write
\begin{equation}\label{eqpre}
\begin{array}{l}
N(V_n(\Z)\cap V_{n,k};X) \\[.1in]
\displaystyle{\;\;\;=\frac1{C_{G_0}^{n,k}}
\int_{\lambda=(\sqrt{3}/2)^{1/2}/C}^{X^{1/2n}}
\int_{t=(\sqrt3/2)^{1/2}}^{C\lambda} \int_{N'(t)}
(\Vol(B(u,t,\lambda,X)) +
O(\max\{{t^n}{\lambda^{n^2}},1\})) t^{-2} 
du \,d^\times t\, d^\times \lambda}\\[.2in]
\qquad\qquad\qquad\qquad\qquad\qquad\qquad\qquad\qquad\qquad\qquad\qquad
+\,O(X^{{\textstyle{\frac{n}{2}}}+\varepsilon}),
\end{array}
\end{equation}
where the latter error term is
 due to the estimate on reducible forms in Lemma~\ref{reducible}.

Let us first consider the evaluation of the integral of the second summand in (\ref{eqpre}).
First, we observe that $t^{n} \lambda^{n^2} \gg 1$,
so that the integral of the second summand is bounded from above by (a constant factor times) 
\begin{equation}\label{eqfirst}
\frac1{C_{G_0}^{n,k}}
\int_{\lambda=(\sqrt{3}/2)^{1/2}/C}^{X^{1/2n}}
\int_{t=(\sqrt3/2)^{1/2}}^{C\lambda} \int_{N'(t)}
t^{n} \lambda^{n^2} t^{-2} dn\, d^\times t\, d^\times \lambda\ll 
\int_{\lambda=(\sqrt3/2)^{1/2}/C}^{X^{1/2n}}
\lambda^{n^2} t^{n-2} 
\Big |_{t = (\sqrt{3}/2)^{1/2}}^{t = C \lambda} \, d^{\times} \lambda 
\end{equation}
\begin{equation}\label{eqsecond}
\ll  \int_{\lambda=(\sqrt3/2)^{1/2}/C)}^{X^{1/2n}} C^{n-2} \lambda^{n^2 + n - 2} d^{\times} \lambda = 
C^{n-2} \lambda^{n^2 + n - 2} \Big|_{\lambda = (\sqrt{3}/2)^{1/2}/C}^{\lambda=X^{1/2n}} = O(X^{\textstyle\frac{n+1}{2} \!-\! \frac{1}{n}}). 
\end{equation}

Meanwhile, the integral of the first summand is
\begin{equation}\label{eq16}
\frac1{C_{G_0}^{n,k}}\int_{h\in G_0}{\rm Vol}
(\mathcal R_X(hL)){dh} -O \Bigl( \int_{\lambda=(\sqrt3/2)^{1/2}/C}^{X^{1/2n}}
\int_{t=C\lambda}^{\infty}\int_{N'(t)}
\Vol(B(u,t,\lambda,X))t^{-2} 
dn \,d^\times t \,d^\times \lambda\Bigr).
\end{equation}
However, ${\rm Vol}(\mathcal R_X(hL))$ is independent of $h$, so that the first term in (\ref{eq16}) is simply
${\rm Vol}(\mathcal R_X(L))$.  Next, using the fact that
${\rm Vol}(B(u,t,\lambda,X))=O(\lambda^{n(n+1)})$, and carrying out
the integration in the second term of (\ref{eq16}) exactly as in (\ref{eqfirst})--(\ref{eqsecond}),
we find that this term is also $O(X^{\frac{n+1}{2} - \frac{1}{n}})$:
\begin{equation}\label{eqfirst2}
\frac1{C_{G_0}^{n,k}}
\int_{\lambda=(\sqrt3/2)^{1/2}/C}^{X^{1/2n}}
\int_{t=C\lambda}^\infty \int_{N'(t)}
\lambda^{n(n+1)} t^{-2} dn\, d^\times t\, d^\times \lambda\ll 
-\int_{\lambda=(\sqrt3/2)^{1/2}/C}^{X^{1/2n}}
\lambda^{n^2+n} t^{-2} 
\Big |_{t = C \lambda}^{t = \infty} \, d^{\times} \lambda 
\end{equation}
\begin{equation}\label{eqsecond2}
\ll \int_{\lambda=(\sqrt3/2)^{1/2}/C}^{X^{1/2n}} C^{-2} \lambda^{n^2 + n - 2} d^{\times} \lambda = 
C^{-2} \lambda^{n^2 + n - 2} \Big|_{(\sqrt3/2)^{1/2}/C}^{X^{1/2n}} = O(X^{\textstyle\frac{n+1}{2} \!-\! \frac{1}{n}}). 
\end{equation}

We conclude that
\begin{equation}\label{bigint}
N_{n,k}(X)={\rm Vol}(\mathcal R_X(L))+O(X^{\textstyle\frac{n+1}{2}\!-\!\frac1n}).
\end{equation}
This proves Theorem \ref{refbq}.

\begin{remark}{\em \label{quadraticremark} 
The proof we have given for $n\geq 3$ also adapts easily to the case      
$n=2$, $k=1$.  Indeed, rather than being finite, the stabilizer in $\SL_2(\R)$     
of a definite binary quadratic form is compact and conjugate to $K$.  In the         
usual way, we may then replace occurrences of cardinalities of sets of group elements        
with integrals over $K$, e.g., $\#\Stab_{\SL_2(\R)}(v)$ is replaced by   
$\int_{\kappa\in K} d\kappa$ (which we may normalize to be 1).  All other            
arguments then hold without any essential change, yielding Theorem~1 for $n=2$ as    
well.}
\end{remark}

\section{Congruence conditions}\label{bfcong}

We may prove a version of Theorem~\ref{refbq} for a set in
$V_{n,k}(\Z)$ that is defined by a finite number of congruence conditions:

\begin{theorem}\label{cong}
  Suppose $S$ is an   $\SL_2(\Z)$-invariant 
 subset of $V_{n,k}(\Z)$ that is defined by 
  congruence conditions modulo finitely many prime powers.
Then we have
\begin{equation}\label{ramanujan}
N(S;X)
  = c_{n,k}\cdot  \prod_{p} \mu_p(S)\cdot X^{\textstyle \frac{n+1}{2}} + O_{S}(X^{\textstyle\frac{n+1}{2}\!-\!\frac1n}),
\end{equation}
where $\mu_p(S)$ denotes the density of the $p$-adic closure of $S$ in $V_n(\Z_p)$.
\end{theorem}

To obtain Theorem~\ref{cong}, note that the set
$S\subset V_{n,k}(\Z)$ in Theorem~\ref{cong} may be
viewed for some fixed integer $m$ as the intersection of $V_{n,k}$ with the union $U$ of (say)
$\tau$ translates $L_1,\ldots,L_\tau$ of the lattice $m\cdot V_n(\Z)$.  For
each such lattice translate $L_j$, we may use formula (\ref{avg}) and
the discussion following that formula to compute $N(L_j\cap
V_{n,k};X)$, where each $d$-dimensional volume is scaled by a factor
of $1/m^d$ to reflect the fact that our new lattice has been scaled by
a factor of $m$.  Proceeding as in \S6 
then gives by the identical arguments:
\begin{equation}\label{sestimate3}
N(S;X) = \tau m^{-(n+1)}{\Vol(\RR_X(v))} + O_S(X^{\textstyle\frac{n+1}{2}\!-\!\frac1n}).
\end{equation}

\noindent
Finally, the identity $\tau m^{-(n+1)}=\prod_p\mu_p(S)$ yields (\ref{ramanujan}).  

\subsection*{Acknowledgments}
We thank John Cremona, Peter Sarnak, Arul Shankar, and Michael Stoll for helpful conversations. This work was done
in part while the authors were at MSRI during the special semester on Arithmetic Statistics.
The first author was supported by a Simons Investigator Grant and NSF grant~DMS-1001828.


\end{document}